\documentclass[12pt,reqno]{amsart}

\usepackage{amssymb}
\usepackage{dsfont}
\usepackage{tikz}
\usetikzlibrary{arrows}
\usepackage[ruled,vlined,algosection,norelsize]{algorithm2e}

\numberwithin{equation}{section}

\def\ZZ{{\mathds Z}}
\def\RR{{\mathds R}}
\def\NN{{\mathds N}}

\def\Stdepth{\operatorname{Stdepth}}
\def\Hdepth{\operatorname{Hdepth}}
\def\depth{\operatorname{depth}}

\def\min{\operatorname{min}}
\def\Ann{\operatorname{Ann}}

\let\Dirsum=\bigoplus

\let\iso=\cong

\newtheorem{lemma}{Lemma}[section]
\newtheorem{coro}[lemma]{Corollary}
\newtheorem{theo}[lemma]{Theorem}
\newtheorem{propo}[lemma]{Proposition}

\theoremstyle{definition}
\newtheorem{defi}[lemma]{Definition}
\newtheorem{rema}[lemma]{Remark}
\newtheorem{ex}[lemma]{Example}
\newtheorem{question}[lemma]{Question}
\newtheorem{prob}[lemma]{Problem}
\title{How to compute the multigraded Hilbert depth of a module}

\author{Bogdan Ichim}
\address{``Simion Stoilow'' Institute of Mathematics of the Romanian Academy, 010702 Bucharest, Romania}
\email{bogdan.ichim@imar.ro}

\author{Julio Jos\'e Moyano-Fern\'andez}
\address{Universit\"at Osnabr\"uck, Institut f\"ur Mathematik, 49069 Osnabr\"uck, Germany}
\email{jmoyanof@uos.de}

\begin{document}

%\thanks{The second author was partially supported by the Spanish Government Ministerio de Educaci\'on y Ciencia (MEC), grant MTM2007-64704 in cooperation with the European Union in the framework of the founds ``FEDER'', and by the Deutsche Forschungsgemeinschaft (DFG)}

\maketitle
\begin{abstract}
In the first part of this paper we introduce a method for computing Hilbert decompositions  (and consequently the Hilbert depth) of a finitely generated multigraded module $M$ over the polynomial ring $K[X_1, \ldots, X_n]$ by reducing the problem to the computation of the finite set of the new defined Hilbert partitions. In the second part we show how Hilbert partitions may be used for computing  the Stanley depth of the module $M$. In particular, we answer two open questions posed by Herzog in \cite{H}.
% \subclass{05E40 \and 16W50 \and 13D40}
\end{abstract}

\section{Introduction}

In this paper we study methods for computing two algebraic-combinatorial invariants, namely the Hilbert depth and the Stanley depth of a finitely generated multigraded module $M$ over the standard multigraded polynomial ring $R=K[X_1,\dots,X_n]$. In particular, we give satisfactory answers to the following two open questions posed by Herzog in \cite{H}:

\begin{question}\label{Q:Herzog}\cite[Question 1.65]{H} Does there exist an algorithm to compute the Stanley
depth of finitely generated multigraded $R$-modules?
\end{question}

\begin{prob}\label{P:Herzog}\cite[Problem 1.66]{H} Find an algorithm to compute the Stanley depth for finitely
generated multigraded $R$-modules $M$ with $\dim_K M_a \le 1$ for all $a\in \ZZ^n$.
\end{prob}

In recent years, \emph{Stanley decompositions} of multigraded modules
over $R$ have been discussed
intensively. Such decompositions, introduced by Stanley in \cite{S},
break the module $M$ into a direct sum of \emph{Stanley spaces},
each being of type $mS$ where $m$ is a homogeneous element of $M$,
$S=K[X_{i_1},\dots,X_{i_d}]$ is a polynomial subalgebra of $R$ and $S\bigcap \Ann m=0$. One says that $M$ has \emph{Stanley depth} $s$, $\Stdepth M=s$,  if
one can find a Stanley decomposition in which $d\ge s$ for each
polynomial subalgebra involved, but none with $s$ replaced by $s+1$.
\medskip

The computation of the Stanley depth is not an easy task, due mainly to its combinatorial  nature. A first step was done by Herzog, Vladoiu and Zheng in \cite{HVZ}, where they introduced a method for computing the Stanley depth of a monomial ideal of $R$. Some remarkable results in the study of the Stanley depth in the multigraded case were presented by Apel (see \cite{A1}, \cite{A2}), Herzog et al.\
(see \cite{HP}, \cite{HSY}) and Popescu et al.\ (see \cite{AP}, \cite{P}).
\medskip

Hilbert series are the most important numerical invariants of finitely generated graded and multigraded modules over $R$ and they form the bridge from commutative algebra to its combinatorial applications (we refer here to classical results of Hilbert, Serre, Ehrhart and Stanley). A new type of decompositions for multigraded modules $M$ depending only on the Hilbert series of $M$ was introduced by Bruns, Krattenthaler and Uliczka in \cite{BKU}: the \emph{Hilbert decompositions}. They are a weaker type of decompositions not requiring
the summands to be submodules of $M$, but only vector subspaces isomorphic to polynomial subrings. The notion of \emph{Hilbert depth} $\Hdepth M$ is defined accordingly. Several results concerning both the graded and multigraded cases were presented in Bruns, Krattenthaler and Uliczka \cite{BKU2}, \cite{U} and Uliczka and the second author \cite{JU}. All of them are based on both combinatorial and algebraic  techniques.
\medskip

The contain of the paper is organized as follows. Section \ref{StanHilb} is devoted to introduce the definitions and main tools concerning Hilbert depth to be used along the paper.
In Section \ref{AlgHilb}, a procedure for computing the Hilbert depth is presented. Remark that the method described in \cite{HVZ} may be used to compute the Hilbert depth in the particular case of a monomial ideal of $R$ (notice that---according to \cite{BKU}---the Hilbert depth coincides with the Stanley depth in that case). By introducing the concept of \emph{Hilbert partition} (cf. Definition \ref{defi:Hpartition}) and using the functorial techniques exposed by E. Miller in \cite{M}, we extend it to a method for computing the Hilbert depth in the general case of a multigraded $R$-module (see Theorem \ref{theo:main} and Corollary \ref{cor:hdepth}).
\medskip

Hilbert decompositions are intimately related to Stanley decompositions: All Stanley decompositions are Hilbert decompositions. In the rest of the paper, we investigate how strong connected are the Hilbert depth and the Stanley depth.
\medskip

In Section \ref{AlgStanley} we present an approach to the problem of computing Stanley depth of a finitely generated multigraded module $M$ over the polynomial ring $R$ based on Section \ref{AlgHilb}. We show that in a \emph{finite} number of steps one can decide whether a Hilbert partition (together with a finite set of elements of $M$) is inducing a Stanley decomposition or not (the converse is always true: any Stanley decomposition induces a Hilbert partition), see Proposition \ref{hilbert:stanley}.

We conclude that there exists an method (although not easy to use) to compute the Stanley depth by looking at all the Hilbert partitions and selecting those that are also inducing Stanley decompositions in Corollary \ref{coro:last1}. Thus, the answer to Question \ref{Q:Herzog} is ``Yes". Moreover, with the assumption that $\dim_K M_a \le 1$ for all $a\in \ZZ^n$, checking whether Hilbert decompositions are inducing Stanley decompositions is easy and we introduce a precise algorithm for computing the Stanley depth in that case in Corollary \ref{coro:last2}. This solves Problem \ref{P:Herzog}.
\medskip

In the last section we show that the methods introduced in the Sections \ref{AlgHilb} and \ref{AlgStanley} can effectively be applied in order to deduce some simple statements (which to the best of our knowledge
are not known to have a proof by using other methods).

\section{Prerequisites}\label{StanHilb}

Throughout the paper we will use the notation $a=(a_1, \ldots, a_n)$ for elements $a \in \mathds{Z}^n$ (or $\mathds{N}^n$). We consider the polynomial ring $R=K[X_1,\dots,X_n]$ over a field
$K$ with the  \emph{multigraded} structure on $R$, namely the $\ZZ^n$-grading
in which the degree of $X_i$ is the $i$-th vector $e_i$ of the
canonical basis of $\RR^n$. For any $c \in \mathds{N}^n$ we will denote $X^{c}:=X_1^{c_1} X_2^{c_2} \cdots X_n^{c_n}$. All $R$-modules we consider are assumed to belong to the category $\mathcal{M}$ of finitely generated $\ZZ^n$-graded (or multigraded) $R$-modules. All the isomorphisms occurring in the paper are in the category of $\ZZ^n$-graded vector spaces. If we need to consider an isomorphism of $R$-modules, we will mention it explicitly.
\medskip

Hilbert functions are one of the most important numerical invariants of graded and multigraded modules; they form the bridge from commutative algebra to its combinatorial applications.  Let $M=\bigoplus_{a \in \mathds{Z}^n} M_a \in \mathcal{M}$. Then we consider its \emph{Hilbert function}
\[
H(M,-) : \ZZ^n \longrightarrow \ZZ, \quad H(M,a)=\dim_K M_a.
\]
For further details about Hilbert functions in the multigraded case the reader is referred to Bruns and Gubeladze \cite{BG}.
\medskip

From the combinatorial viewpoint a module is often only an algebraic
substrate of its Hilbert function, and one may ask which
presentation a given Hilbert function can have. Following \cite{BKU} we define the main objects of our study, namely Hilbert decompositions and Hilbert depth of modules.

\begin{defi}
A \emph{Hilbert
decomposition of} $M$ is a finite family
\[
\mathfrak{D} : (R_i,s_i)_{i\in I}
\]
such that $R_i$ are subalgebras generated by a subset of the indeterminates of
$R$ for each $i\in I$, $s_i\in \ZZ^n$, and
\[
M\iso\Dirsum_{i\in I} R_i(-s_i)
\]
as a multigraded $K$-vector space.
\end{defi}

Observe that all the Hilbert decompositions of a module $M$ depend only on
the Hilbert function of $M$.

\begin{defi}
A Hilbert decomposition carries the structure of an $R$-module and has a well-defined depth, which is called the \emph{depth of the Hilbert decomposition} $\mathfrak{D}$ and will be denoted by $\depth \mathfrak{D}$.
The \emph{Hilbert depth} of a module $M$ is
\[
\max\{\depth \mathfrak{D} \ | \ \mathfrak{D} \text{ is a Hilbert decomposition of } M\}
\]
 and will be denoted by $\Hdepth M$.
\end{defi}

Next, we shall consider a natural partial order on $\ZZ^n$ as follows: Given $a,b\in \ZZ^n$, we say that $a \preceq b$ if and only if $a_i\le b_i$ for $i=1,\ldots,n$. Note that $\ZZ^n$ with this partial order is a distributive lattice with meet $a\wedge b$ and join $a\vee b$ being the componentwise minimum and maximum, respectively.
We set the interval between $a$ and $b$ to be
\[
[a,b]:=\{c \in \ZZ^n \ |\  a \preceq c \preceq b \}.
\]

We recall some definitions and results given by E. Miller in \cite{M} which will be useful in the sequel.
Let $g\in \mathds{N}^n$. The module $M$ is said to be $\mathds{N}^n$-graded if $M_a=0$ for $a \notin \mathds{N}^n$; $M$ is said to be positively $g$-determined if it is $\mathds{N}^n$-graded and the multiplication map $\cdot X_i : M_{a} \longrightarrow M_{a+e_i}$ is an isomorphism whenever $a_i \ge g_i$.
A characterization of positively $g$-determined modules is given by the following.

\begin{propo}\label{prop:ezra}\cite[Proposition 2.5]{M}
The module $M \in \mathcal{M}$ is positively $g$-determined if and only if the multigraded Betti numbers of $M$ satisfy $\beta_{0,a}=\beta_{1,a}=0$ unless $0 \preceq a \preceq g$.
\end{propo}

Important tools also introduced in \cite[Definition 2.7]{M} are the following functors:
\begin{itemize}
\item[(1)] the subquotient bounded in the interval $[0,g]$, denoted by $\mathcal{B}_{g}$, where
\[
\mathcal{B}_{g}(M):= \bigoplus_{0 \preceq a \preceq g} M_{a};
\]
\item[(2)] the positive extension of $M$, denoted by $\mathcal{P}_{g}$, where
\[
\mathcal{P}_{g}(M):=\bigoplus_{a \in \mathds{Z}^n}M_{g \wedge a}
\]
or, in other words, $(\mathcal{P}_{g}M)_{a}=M_{g \wedge a}$, endowed with the $R$-action
\[
(\cdot X_i)_g: (\mathcal{P}_{g}M)_{a} \to (\mathcal{P}_{g}M)_{a+e_i}
\]
defined as the multiplication map $\cdot X_i: M_{g \wedge a} \to M_{g \wedge a+e_i}$ if $a_i < g_i$; or as the identity map otherwise.
\end{itemize}

\medskip

From the above definitions one can immediately obtain:

\begin{propo}\label{prop:1}\cite[Theorem 2.11]{M} Let $g\in \mathds{N}^n$, and assume that $M \in \mathcal{M}$ is positively $g$-determined. Then
$$
\mathcal{P}_{g}(\mathcal{B}_{g}(M))\iso M.
$$
\end{propo}

The following example makes clear the behaviour of the functors $\mathcal{B}_{g}$ and $\mathcal{P}_{g}$.

\begin{ex}\label{ex:1}\cite[Example 2.8]{M} Let $a \in \mathds{N}^n$. We have $\mathcal{B}_g(R(-a))=0$ unless $a \preceq g$, in which case we have that
\[
\mathcal{B}_g(R(-a)) \cong (R/\langle X_1^{g_1+1-a_1}, \ldots , X_n^{g_n+1-a_n} \rangle)(-a)
\]
is the artinian subquotient of $R$ which is nonzero precisely in the degrees from the interval $[a,g]$. Applying $\mathcal{P}_g$ to this yields back $R(-a)$ so that
$\mathcal{P}_{g} (\mathcal{B}_{g}(R(-a)))$ is isomorphic to $R(-a)$ if $a \preceq g$.
\end{ex}

Figure \ref{fig1}, in which the circles represent the graded components of $\mathcal{B}_{g}(M)$, illustrates the action of the functor $\mathcal{P}_{g}$.

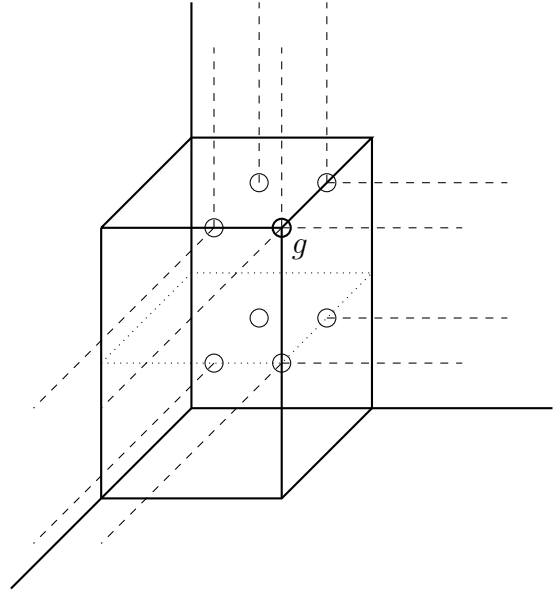
\begin{figure}[h]\label{fig:one}
\begin{center}
  \begin{tikzpicture}[scale=0.6]
    \draw[thick] (0,0) -- (4,4) -- (12,4);
    \draw[thick] (4,4) -- (4,13);
    \draw[thick] (2,2) -- (6,2) -- (8,4);
    \draw[thick] (2,8) -- (6,8) node[sloped,below, pos=1.1] {$g$} -- (8,10) -- (4,10) -- (2,8);
    \draw[thick] (6,2) -- (6,8);
    \draw[thick] (2,2) -- (2,8);
    \draw[thick] (8,10) -- (8,4);

    \draw[dashed] (6,5) -- (10,5);
    \draw[dashed] (6,5) -- (2,1);

    \draw[dashed] (7,6) -- (11,6);
    \draw[dashed] (4.5,5) -- (0.5,1);

    \draw[dashed] (6,8) -- (10,8);
    \draw[dashed] (6,8) -- (2,4);
    \draw[dashed] (6,8) -- (6,12);

   \draw[dashed] (4.5,8) -- (4.5,12);
   \draw[dashed] (4.5,8) -- (0.5,4);
   \draw[dashed] (5.5,9) -- (5.5,13);

     \draw[dashed] (7,9) -- (11,9);
     \draw[dashed] (7,9) -- (7,13);

    \draw[dotted] (2,5) -- (6,5) -- (8,7) -- (4,7) -- (2,5);

    \draw(7,6) circle(0.2);
    \draw(6,5) circle(0.2);
    \draw[thick](6,8) circle(0.2);
    \draw(4.5,5) circle(0.2);
    \draw(4.5,8) circle(0.2);
    \draw(7,9) circle(0.2);
     \draw(5.5,6) circle(0.2);
    \draw(5.5,9) circle(0.2);
   \end{tikzpicture}
\caption{The functor $\mathcal{P}_{g}$} \label{fig1}
\end{center}
\end{figure}

\section{A method for computing the multigraded Hilbert depth of a module}\label{AlgHilb}

The aim of this section is to describe a procedure for computing the Hilbert depth of a multigraded module over the polynomial ring.
Let $M$ denote a finitely generated  multigraded $R$-module with a minimal multigraded free presentation
\[
\bigoplus_{a\in \ZZ^n} R(-a)^{\beta_{1,a}} \longrightarrow \bigoplus_{a\in \ZZ^n} R(-a)^{\beta_{0,a}} \longrightarrow M \longrightarrow 0,
\]
and assume for simplicity, and without loss of generality, that all $\beta_{0,a}=0$  (and \emph{a fortiori}
all $\beta_{1,a}=0$) if $a \notin \NN^n$.
\medskip

We shall also consider the \emph{Hilbert series} of $M$, that is

\[
H_M(X)=\sum_{a \in \mathds{N}^n} H(M,a)X^a.
\]

Let $g\in \NN^n$ be such that the multigraded Betti numbers of $M$ satisfy the equalities $\beta_{0,a}=\beta_{1,a}=0$ unless $0\preceq a \preceq g$. Then, according to Proposition \ref{prop:ezra}, the module $M$ is positively $g$-determined. The Hilbert series of $M$ can be recovered from the polynomial
\[
H_M(X)_{\preceq g}:= H_{\mathcal{B}_g(M)}(X)=\sum_{\substack{0\preceq  a \preceq g}} H(M,a)X^a
\]
since, according to Proposition \ref{prop:1}, we have that $\mathcal{P}_{g}(\mathcal{B}_{g}(M))=M$. This fact may be used in order to actually compute the Hilbert depth of $M$, as presented in this section.
\medskip

Given $a,b \in \ZZ^n$ such that $a\preceq b$, we set
\[
Q[a,b](X):=\sum_{a \preceq c \preceq b } X^{c}
\]
and call it the \emph{polynomial induced by the interval} $[a,b]$.

\begin{defi}\label{defi:Hpartition}
We define a \emph{Hilbert partition}  of the polynomial $H_M(X)_{\preceq g}$ to be an expression
\[
\mathfrak{P}: H_M(X)_{\preceq g}=\sum_{i \in I_{\mathfrak{P}}} Q[a^i,b^i](X)
\]
as a finite sum of polynomials induced by the intervals $[a^i,b^i]$ (the notation $I_{\mathfrak{P}}$ makes clear the dependency on
$\mathfrak{P}$ and so the finiteness).
\end{defi}

In order to describe the Hilbert decomposition of $M$ induced by the Hilbert partition $\mathfrak{P}$  of $H_M(X)_{\preceq g}$, we introduce the following notations. For  $a\preceq g$ we set $Z_{a}=\{X_j \ |\  a_j=g_j\}$. Moreover we denote by $K[Z_{a}]$ the subalgebra generated by the subset of the indeterminates $Z_{a}$. We also define the map
\[
\rho:\{0\preceq a\preceq g\} \longrightarrow \NN, \quad \rho (a):=|Z_a|,
\]
and for $0 \preceq a \preceq b\preceq g$ we set
\[
\mathcal{G}[a,b]:=\{c\in [a,b] \mid c_j=a_j \text{ for all } j \in \mathds{N} \text{ with } b_j=g_j\}.
\]

\bigskip
\begin{lemma}\label{lemma:interval}Let $0 \preceq a \preceq b\preceq g$. Set $K[a,b]=\mathcal{B}_b(R(-a))$.
Then $$\mathcal{P}_g(K[a,b])=\bigoplus_{c\in \mathcal{G}[a,b]} K[Z_{b}](-c)$$
is a Hilbert decomposition of $\mathcal{P}_g(K[a,b])$.
\end{lemma}

\begin{proof} We have
\[
\mathcal{B}_b(R(-a))= \mathcal{B}_b(\bigoplus_{c\in \mathcal{G}[a,b]} K[Z_{b}](-c))=\mathcal{B}_g(\bigoplus_{c\in \mathcal{G}[a,b]} K[Z_{b}](-c)).
\]
Since $\mathcal{B}_g$ and $\mathcal{P}_g$ are $K$-linear functors, the conclusion follows from Proposition \ref{prop:1} by applying $\mathcal{P}_g$.
\end{proof}

We can now state the main theorem of this paper.

\begin{theo} \label{theo:main} The following statements hold:
\begin{enumerate}
\item Let $\mathfrak{P}:H_M(X)_{\preceq g}=\sum_{i=1}^r Q[a^i, b^i](X)$ be a Hilbert partition of $H_M(X)_{\preceq g}$. Then
\[
\mathfrak{D}(\mathfrak{P}): M\iso \bigoplus_{i=1}^r\Big(\bigoplus_{c\in \mathcal{G}[a^i,b^i]} K[Z_{b^i}](-c)\Big)  \hspace{8.3cm}[\star]
\]
is a Hilbert decomposition of $M$. Moreover,
\[
\Hdepth \mathfrak{D}(\mathfrak{P})=\min\{\rho(b^i):\  i=1,\ldots,r\}.
\]
\item Let $\mathfrak{D}$ be a Hilbert decomposition of $M$. Then there exists a Hilbert partition $\mathfrak{P}$ of $H_M(X)_{\preceq g}$ such that
\[
\Hdepth \mathfrak{D}(\mathfrak{P})\ge \Hdepth \mathfrak{D}.
\]
In particular, $\Hdepth M$ can be computed as the maximum of the numbers $\Hdepth \mathfrak{D}(\mathfrak{P})$, where $\mathfrak{P}$ runs over the finitely many Hilbert partitions of $H_M(X)_{\preceq g}$.
\end{enumerate}
\end{theo}

\begin{proof}

\bigskip
(1) The Hilbert partition $\mathfrak{P}$ of $H_M(X)_{\preceq g}$ induces a decomposition
\[
\mathcal{B}_g(M)=\bigoplus_{i=1}^{r} K[a^i,b^i]
\]
of $\mathcal{B}_g(M)$ as a direct sum of subquotients of $R$ bounded in the interval $[0,g]$ and seen as $K$-vector spaces. Since $\mathcal{P}_g$ is a $K$-linear functor, Proposition \ref{prop:1}  yields the decomposition
\[
M\iso \mathcal{P}_g(\mathcal{B}_g(M))=\bigoplus_{i=1}^{r} \mathcal{P}_g(K[a^i,b^i]).
\]
By Lemma \ref{lemma:interval}, we obtain the desired decomposition $[\star]$. The statement about the Hilbert depth of $\mathfrak{D}(\mathfrak{P})$ follows straight from the definitions. This proves the statement (1).

\bigskip
(2) Let $T=K[Z](-a)$ be a Hilbert space. Then we have
\[
 \mathcal{B}_g(T) = \left\{
        \begin{tabular}{cc}
        	$K[a,b(a)]$ & if $a \preceq g$; \\
        	$0$ & otherwise, \\
        \end{tabular}
\right.
\]
where the components of $b(a) \in \mathds{N}^n$ are defined as
\[
 b(a)_j= \left\{
        \begin{tabular}{cc}
        	$a_j$ & if $X_j \notin Z$; \\
        	$g_j$ & otherwise.\\
        \end{tabular}
\right.
\]
In particular $\rho (b(a)) \ge |Z|$. Therefore, if $\mathfrak{D}: M\iso \bigoplus_{i=1}^{r} K[Z_i](-a^i)$ is a Hilbert decomposition of $M$, then
\[
\mathcal{B}_g(M)=\bigoplus_{i=1}^{r} \mathcal{B}_g(K[Z_i](-a^i)) = \bigoplus_{\substack{i \\ a^i \preceq g}} K[a^i,b(a)^i]
\]
where $\rho(b(a)^i) \ge |Z_i|$ for all $i$ such that $a^i \preceq g$. It follows that
\[
\mathfrak{P}:H_M(X)_{\preceq g}=\sum_{i=1}^r Q[a^i, b(a)^i](X)
\]
is a Hilbert partition of $H_M(X)_{\preceq g}$, and (1) implies the inequality $\Hdepth \mathfrak{D}(\mathfrak{P})\ge \Hdepth \mathfrak{D}$.
\end{proof}

It is now an easy matter to check:

\begin{coro}\label{cor:hdepth}
Let $M$ a finitely generated multigraded $R$-module. Then
\[
\Hdepth M =\max \{\Hdepth \mathfrak{D}(\mathfrak{P}): \mathfrak{P} \mbox{ is a Hilbert partition of } H_M(X)_{\preceq g}\}.
\]
In particular, there exists a Hilbert partition $\mathfrak{P}:H_M(X)_{\preceq g}=\sum_{i=1}^r Q[a^i, b^i](X)$ of $H_M(X)_{\preceq g}$
such that
\[
\Hdepth M =\min\{\rho(b^i):\  i=1,\ldots,r\}.
\]
\end{coro}

We finish this section with two examples which show that Corollary \ref{cor:hdepth} can be used in an effective way for computing $\Hdepth M$.

\begin{ex}\label{example1}
Let $R=K[X_1,X_2]$ with $\deg (X_1)=(1,0)$ and $\deg (X_2)=(0,1)$. Let $M=R \oplus (X_1,X_2)R$. First of all, a minimal multigraded free resolution of $M$ is obtained by adding
a minimal multigraded free resolution of $R$
\[
0 \longrightarrow R(-(0,0)) \longrightarrow R(-(0,0)) \longrightarrow 0
\]
and a minimal multigraded free resolution of $(X_1,X_2)R$, namely
\[
0 \longrightarrow R(-(1,1)) \longrightarrow R(-(0,1)) \oplus R(-(1,0)) \longrightarrow (X_1,X_2)R  \longrightarrow 0.
\]
This shows that we may choose $g=(1,1)$. A simple inspection to the shape of $M$ shows that
\[
H_M(X_1,X_2)_{\preceq (1,1)}=1+2X_1+2X_2+2X_1X_2.
\]
It is easy to check that there are no Hilbert partitions containing only monomials of degree two as right ends of the intervals.
The Hilbert partitions containing monomials of degree $\ge 1$ as right ends of the intervals are
\begin{align*}
\mathfrak{P}_1:& (1+X_1+X_2+X_1X_2)+(X_1+X_1X_2)+X_2,\\
\mathfrak{P}_2:& (1+X_1+X_2+X_1X_2)+(X_2+X_1X_2)+X_1,\\
\mathfrak{P}_3:& (1+X_1+X_2+X_1X_2)+X_1+X_2+X_1X_2,\\
\mathfrak{P}_4:& (1+X_1)+(X_1+X_1X_2)+2X_2+X_1X_2,\\
\mathfrak{P}_5:& (1+X_1)+(X_1+X_1X_2)+X_2+(X_2+X_1X_2),\\
\mathfrak{P}_6:& (1+X_1)+2(X_2+X_1X_2)+X_1,\\
\mathfrak{P}_7:& (1+X_1)+(X_2+X_1X_2)+X_1+X_2+X_1X_2,\\
\mathfrak{P}_8:& (1+X_1)+X_1+2X_2+2X_1X_2,\\
\mathfrak{P}_9:& (1+X_2)+(X_2+X_1X_2)+2X_1+X_1X_2,\\
\mathfrak{P}_{10}:& (1+X_2)+2(X_1+X_1X_2)+X_2,\\
\mathfrak{P}_{11}:& (1+X_2)+(X_1+X_1X_2)+(X_2+X_1X_2)+X_1,\\
\mathfrak{P}_{12}:& (1+X_2)+(X_1+X_1X_2)+X_1+X_2+X_1X_2,\\
\mathfrak{P}_{13}:& (1+X_2)+2X_1+X_2+2X_1X_2.
\end{align*}

We see also that $\mathrm{Hdepth}(M)=1$. In the sequel we will focus on the Hilbert partitions $\mathfrak{P}_1$ and $\mathfrak{P}_3$.
They are represented in Figure \ref{fig:two} where the monomials are indicated by $\circ$, and the corresponding coefficients by numbers with an arrow pointing at the circles.
\begin{figure}[h]
\begin{center}
  \begin{tikzpicture}[scale=0.6]
    \draw[thick] (0,7) -- (0,0) -- (7,0) ;
   \draw[] (0,5) -- (5,5) -- (5,0);

    \node [above right] at (0,0) {$\swarrow 1$};
    \node [above left] at (4.5,0) {$ 2\searrow $};
    \node [below right] at (0,5) {$\nwarrow 2 $};
    \node [below left] at (4.5,5) {$ 2 \nearrow$};

    \draw [ultra thick,rounded corners] (-1,-1) rectangle (6,6);

    \draw(5,5) circle(0.2);
    \draw(0,0) circle(0.2);
    \draw(5,0) circle(0.2) ;
    \draw(0,5) circle(0.2);

    \draw[thick](0,5) circle(0.4);
    \draw [thick,rounded corners] (4.5,-0.5) rectangle (5.5,5.5);

   \draw[thick] (10,7) -- (10,0) -- (17,0) ;
   \draw[] (10,5) -- (15,5) -- (15,0);

    \node [above right] at (10,0) {$\swarrow 1$};
    \node [above left] at (15,0) {$2 \searrow$};
    \node [below right] at (10,5) {$\nwarrow 2$};
     \node [below left] at (15,5) {$2 \nearrow$};

    \draw [ultra thick,rounded corners] (9,-1) rectangle (16,6);

    \draw(15,5) circle(0.2);
    \draw(10,0) circle(0.2);
    \draw(15,0) circle(0.2) ;
    \draw(10,5) circle(0.2);
    \draw[thick](15,5) circle(0.4);
    \draw[thick](10,5) circle(0.4);
    \draw[thick](15,0) circle(0.4);
   \end{tikzpicture}
\caption{Hilbert partitions $\mathfrak{P}_1$ and $\mathfrak{P}_3$} \label{fig:two}
\end{center}
\end{figure}
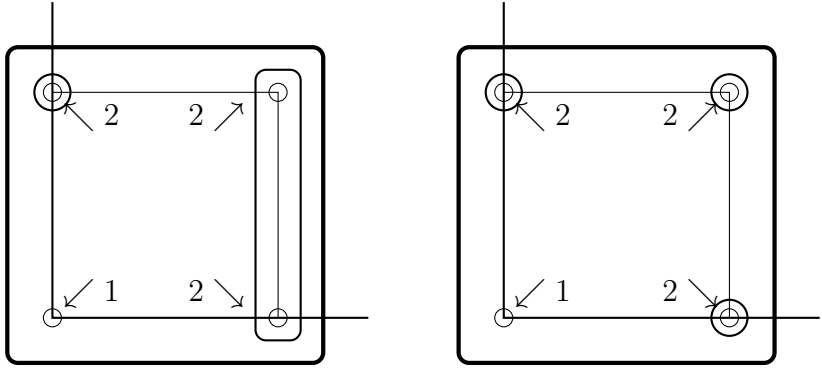

Next we describe the induced Hilbert decompositions. For  $\mathfrak{P}_1$ we have $r=3$ with $[a^1,b^{1}]=[(0,0),(1,1)]$, $[a^2,b^{2}]=[(1,0),(1,1)]$ and $[a^3,b^{3}]=[(0,1),(0,1)]$, and so $Z_{b^{1}}=Z_{b^{2}}=\{X_1,X_2\}$ and $Z_{b^{3}}=\{X_2\}$. Therefore
\begin{align*}
\mathcal{G}([(0,0),(1,1)])=& \{(0,0)\},\\
\mathcal{G}([(1,0),(1,1)])=& \{(1,0)\},\\
\mathcal{G}([(0,1),(0,1)])=& \{(0,1)\}.
\end{align*}

The induced Hilbert decomposition of $M$ is in this case
\begin{eqnarray*}
\mathfrak{D}(\mathfrak{P}_1):  M\iso K[X_1,X_2](-(0,0)) \oplus K[X_1,X_2](-(1,0)) \oplus    K[X_2](-(0,1)).
\end{eqnarray*}

Similarly one gets
\begin{align*}
\mathfrak{D}(\mathfrak{P}_3): M\iso K[X_1,X_2](-(0,0)) \oplus K[X_1,X_2](-(1,1)) \oplus K[X_1](-(1,0)) \oplus K[X_2](-(0,1)).
\end{align*}
\end{ex}

\begin{ex}\label{example2}
Let $R=K[X_1,X_2]$ and $M=K \oplus X_2K[X_2] \oplus X_2K[X_1,X_2]=R/(X_1,X_2)\oplus X_2R/(X_1)\oplus X_2R$.
Similar arguments involving the graded free resolution of $M$ show that one can choose $g=(1,1)$. Then
\[
H_M(X_1,X_2)_{\preceq (1,1)} = 1 + 2X_2 + X_1X_2.
\]
It is easily seen that there are no Hilbert partitions containing only monomials of degree two as right ends of the intervals and
the Hilbert partitions containing monomials of degree $\ge 1$ as right ends of the intervals are
\begin{align*}
\mathfrak{P}_1:& (1+X_2)+(X_2+X_1X_2),\\
\mathfrak{P}_2:& (1+X_2)+X_2+X_1X_2.
\end{align*}

\begin{figure}[h]
\begin{center}
  \begin{tikzpicture}[scale=0.6]
    \draw[thick] (0,7) -- (0,0) -- (7,0) ;
   \draw[] (0,5) -- (5,5) -- (5,0);

    \node [below right] at (0.5,-0.5) {$\nwarrow 1$};
    \node [below right] at (5.5,-0.5) {$\nwarrow 0$};
    \node [below right] at (0.5,4.5) {$\nwarrow 2$};
    \node [below right] at (5.5,4.5) {$\nwarrow 1$};

     \draw [ultra thick,rounded corners] (-0.5,-1) rectangle (0.5,6);
     \draw [ultra thick,rounded corners] (-1,5.5) rectangle (6,4.5);

    \draw(5,5) circle(0.2);
    \draw(0,0) circle(0.2);
    \draw(5,0) circle(0.2);
    \draw(0,5) circle(0.2);

   \end{tikzpicture}
\caption{Hilbert partition $\mathfrak{P}_1$}\label{fig:three}
\end{center}
\end{figure}

They yield the induced Hilbert decompositions
\begin{align*}
\mathfrak{D}(\mathfrak{P}_1): & M\iso K[X_2](-(0,0)) \oplus K[X_1,X_2](-(0,1)),\\
\mathfrak{D}(\mathfrak{P}_2): & M\iso K[X_2](-(0,0)) \oplus K[X_2](-(0,1)) \oplus    K[X_1,X_2](-(1,1)).
\end{align*}
Notice that $\mathrm{Hdepth}~\mathfrak{D}(\mathfrak{P}_1)=\mathrm{Hdepth}~\mathfrak{D}(\mathfrak{P}_2)=1$, and we have  $\mathrm{Hdepth}(M)=1$.

\end{ex}

\section{A method for computing the Stanley depth of a module}\label{AlgStanley}

In this section we shall use Theorem \ref{theo:main} in order to compute the Stanley depth of a finitely generated multigraded $R$-module $M$. For simplicity we shall make the same assumptions as in Section \ref{AlgHilb}. First of all we recall what Stanley depth is.

\begin{defi}
A \emph{Stanley
decomposition of} $M$ is a finite family
\[
\mathfrak{F} : (R_i,m_i)_{i\in I}
\]
in which $m_i$ are homogeneous elements of $M$ and $R_i$ are subalgebras generated by a subset of the indeterminates of
$R$ such that $R_i\bigcap \Ann m_i=0$ for each $i\in I$, and
\[
M=\Dirsum_{i\in I} m_iR_i
\]
as a multigraded $K$-vector space. The multigraded $K$-subspace $m_iR_i\subset M$ is called a \emph{Stanley space}.
\end{defi}

Note that every Stanley decomposition induces a Hilbert decomposition and, in particular, it has a well-defined depth. Therefore the following definition makes sense:

\begin{defi}
The \emph{Stanley depth} of a module $M$ is
\[
\max\{\depth \mathfrak{F} \ | \ \mathfrak{F} \text{ is a Stanley decomposition of } M\}
\]
and will be denoted by $\Stdepth M$.
\end{defi}

\begin{rema} Notice that Stanley depth equals Hilbert depth if $\dim_K M_a \le 1$ for all $a\in \NN^n$ and $R_sM_t\not =0$ whenever
$R_s,M_t, M_{s+t}\not =0$ (see \cite[Proposition 2.8]{BKU}). This is true for example if $M=I/J$ where $J \subset I$ are two monomial ideals. In this particular case, \cite[Theorem 2.1]{HVZ} provides a method to compute $\Stdepth M=\Hdepth M$. The two methods coincide in this case (consider the poset of monomials of $H_M(X)_{\preceq g}$).

However, this result does not extend naturally to a method for computing Stanley depth, what explains the difficulty to compute Stanley depth of a finitely generated multigraded $R$-module.
Theorem \ref{theo:main} fills directly this gap if $\Stdepth M=\Hdepth M$, like in the situation described above.
\end{rema}

In the sequel we present a possible approach to the problem of computing Stanley depth of a finitely generated multigraded $R$-module $M$. The next proposition shows that one can decide in a \emph{finite} number of steps whether a Hilbert decomposition is inducing a Stanley decomposition or not (compare with \cite[Proposition 2.9]{BKU}, where essentially an infinite number of steps is necessary for deciding when a Hilbert decomposition can be converted into a Stanley one).

\begin{propo}\label{hilbert:stanley} Let $\mathfrak{P}:H_M(X)_{\preceq g}=\sum_{i=1}^r Q[a^i, b^i](X)$ be a Hilbert partition of $H_M(X)_{\preceq g}$, and let
\[
\mathfrak{D}(\mathfrak{P}): M\iso\bigoplus_{i=1}^r\Big(\bigoplus_{c\in \mathcal{G}[a^i,b^i]} K[Z_{b^i}](-c)\Big)=\Dirsum_{i\in I} R_i(-s_i)
\]
be the induced Hilbert decomposition of $M$. Remark that $I$ is finite (it depends on $\mathfrak{P}$ and the finite sets $\mathcal{G}[a^i,b^i]$) and that $s_i\preceq g$. For all $i\in I$,  choose $0\not = m_i\in M_{s_i}$. The following statements are equivalent:
\begin{enumerate}
\item The decomposition
\[
M=\Dirsum_{i\in I} m_iR_i
\]
is a Stanley decomposition of M.
\item For all $i\in I$ we have that $R_i\bigcap \Ann m_i=0$, and if
\[
\sum_{i\in I} m_i (\sum_{s_i+t_{i_j}\preceq g} \alpha_{i_j}X^{t_{i_j}})=0
\]
with $\alpha_{i_j}\in K$, $X^{t_{i_j}}\in R_i$, then $\alpha_{i_j}=0$ for all $i_j$.
\end{enumerate}
All the Stanley decompositions induced by suitable choices of elements $m_i$ have the same $\Stdepth$ equal to $\Hdepth \mathfrak{D}(\mathfrak{P})$.
\end{propo}

\begin{proof} We only have to show that (2) implies (1). The condition $R_i\bigcap \Ann m_i=0$ assures that $m_iR_i$ is a Stanley space. In order to prove that the sum in (1) is direct, it suffices to show that any two different Stanley spaces in (1) have no homogeneous element in common.
\medskip

Let $m_s\in M_s$ be a homogeneous element and assume for simplicity that
\[
m_s\in m_{s_1}K[Z_{b^1}]\cap m_{s_2}K[Z_{b^2}],
\]
where $m_{s_1}\in M_{s_1}$, $m_{s_2}\in M_{s_2}$ and $s_1, s_2 \in I$. It is clear that $s_1\preceq s$ and $s_2\preceq s$, and therefore
\[
m_{s_1}  \alpha_{1}X^{t_{1}}= m_{s_2} \alpha_{2}X^{t_{2}},
\]
where $s_1+t_1=s_2+t_2=s$, $\alpha_{1}, \alpha_{2}\in K$, $X^{t_{1}}\in K[Z_{b^1}]$ and $X^{t_{2}}\in K[Z_{b^2}]$. If $s\preceq g$ then (2) implies directly $\alpha_{1}=\alpha_{2}=0$.
\medskip

Next let us suppose $s\not\preceq g$. We have $s_1\preceq s$ and $s_1 \preceq g$, which implies $s_1\preceq s\wedge g$. We claim that $X^{s-s\wedge g}\in  K[Z_{b^1}]$.
If $b^1_l=g_l$ then $X_l\in Z_{b^1}$. Otherwise, if $b^1_l<g_l$ then $X_l\not\in Z_{b^1}$ and hence $s_l=(s_1)_l\le g_l$. It follows $(s\wedge g)_l = s_l$ and hence
$(s-s\wedge g)_l = 0$.
\medskip

Similarly we have $X^{s-s\wedge g}\in  K[Z_{b^2}]$.
Since $M$ is positively $g$-determined, the multiplication map
\[
\cdot X^{s-s\wedge g} : M_{s\wedge g} \longrightarrow M_{s}
\]
is an isomorphism. Hence
\[
m_{s_1}  \alpha_{1}X^{t_{1}-s+s\wedge g}= m_{s_2} \alpha_{2}X^{t_{2}-s+s\wedge g}.
\]
Now it is easily seen that
\[
s_1+t_{1}-s+s\wedge g=s_2+t_{2}-s+s\wedge g=s\wedge g\preceq g
\]
and (2) implies $\alpha_{1}=\alpha_{2}=0$.
\end{proof}

Remark that in general a Hilbert partition will not induce a Stanley decomposition, as the following example shows.

\begin{ex}
Let us consider again the module
\[
M=K\oplus X_2K[X_2]\oplus X_2K[X_1,X_2]=R/(X_1,X_2)\oplus X_2R/(X_1)\oplus X_2R
\]
of Example \ref{example2} (the right side is actually the R-module structure). Then $K[X_2](-(0,0))\oplus K[X_1,X_2](-(0,1))$ is a Hilbert decomposition of $M$ which does not induce a Stanley decomposition $M=m_1K[X_2]\oplus m_2K[X_1,X_2]$. Since $M_{(0,0)}=K$ and every element in $K$ is annihilated by the ideal $(X_1,X_2)$, there is no possible choice for $m_1$. The same holds for the Hilbert decomposition $K[X_2](-(0,0)) \oplus K[X_2](-(0,1)) \oplus    K[X_1,X_2](-(1,1))$. We conclude that $\Stdepth M=0$.
\end{ex}

Proposition \ref{hilbert:stanley} allows us to prove the main result of this section, which shows that the Stanley depth can be computed by looking at the Hilbert partitions.

\begin{theo} \label{theo:section4}
Let $\mathfrak{F}$ be a Stanley decomposition of $M$. Then there exists a Hilbert partition $\mathfrak{P}$ of $H_M(X)_{\preceq g}$ inducing a Hilbert decomposition
\[
\mathfrak{D}(\mathfrak{P}): M\iso\Dirsum_{i\in I} R_i(-s_i)
\]
and $0\not = m_i\in M_{s_i}$ for all $i\in I$, such that the Hilbert decomposition $\mathfrak{D}(\mathfrak{P})$ induces a Stanley decomposition
\[
\overline{\mathfrak{D}(\mathfrak{P})}: M=\Dirsum_{i\in I} m_iR_i
\]
with $\Stdepth \overline{\mathfrak{D}(\mathfrak{P})}\ge \Stdepth \mathfrak{F}$.
\end{theo}

\begin{proof} Let $mK[Z]$ be a Stanley space in $\mathfrak{F}$ such that $m\in M_{a}$. Then we have
\[
 \mathcal{B}_g(mK[Z]) = \left\{
        \begin{tabular}{cc}
        	$K[a,b(a)]$ & if $a\preceq g$; \\
        	$0$ & otherwise, \\
        \end{tabular}
\right.
\]
where
\[
 b(a)_l= \left\{
        \begin{tabular}{cc}
        	$a_l$ & if $X_l \notin Z$; \\
        	$g_l$ & otherwise.\\
        \end{tabular}
\right.
\]
Suppose $a\preceq g$. Then $X_l\in Z_{b(a)}$ only if $X_l\in Z$ or $a_l=g_l$. We have
\[
\Ann m\cap K[Z]=0
\]
since  $mK[Z]$ is a Stanley space. If $a_l=g_l$ then $\Ann m\cap K[Z\cup \{X_l\}]=0$ since $M$ is positively $g$-determined, so the multiplication map $\cdot X_l : M_a \longrightarrow M_{a+e_l}$ is injective. We may replace $Z$ by $Z\cup \{X_l\}$ and after a finite number of steps we deduce that
\[
\Ann m\cap K[Z_{b(a)}]=0.
\]

Remark the following fact:
\begin{equation}
\mbox{If~} X^t \in K[Z_{b(a)}] \mbox{~and~} a+t \preceq g, \mbox{~then~} X^t\in K[Z]. \tag{$\ast$} \label{eq:FI}
\end{equation}

Let $\mathfrak{F}: M=\bigoplus_{i=1}^{r} m_iK[Z_i]$ with $m_i\in M_{a^i}$ be a Stanley decomposition of $M$. Then
\[
\mathfrak{P}:H_M(X)_{\preceq g}=\sum_{a^i\preceq g} Q[a^i, b(a^i)](X)
\]
is a Hilbert partition of $H_M(X)_{\preceq g}$ and $\Ann m_i\cap K[Z_{b(a^i)}]=0$.  Moreover, if
\[
\sum_{a^i\preceq g} m_i (\sum_{a^i+t_{i_j}\preceq g} \alpha_{i_j}X^{t_{i_j}})=0
\]
with $\alpha_{i_j}\in K$, $X^{t_{i_j}}\in K[Z_{b(a^i)}]$, then the fact \eqref{eq:FI} implies that $X^{t_{i_j}}\in K[Z_i]$. It follows that $\alpha_{i_j}=0$ for all $i_j$ since  $\mathfrak{F}$ is a Stanley decomposition. By Proposition \ref{hilbert:stanley} it is easily seen that the induced decomposition
\[
\overline{\mathfrak{D}(\mathfrak{P})}: \sum_{a^i\preceq g} m_i K[Z_{b(a^i)}]
\]
is a Stanley decomposition. Finally, Theorem \ref{theo:main} yields the desired inequality $\Stdepth \overline{\mathfrak{D}(\mathfrak{P})}\ge \Stdepth \mathfrak{F}$.

\end{proof}

A procedure for the computation of the Stanley depth can be simply deduced now.

\begin{coro} \label{coro:last1}
$\Stdepth M$ may be computed by considering the Hilbert decompositions $\mathfrak{D}(\mathfrak{P})$, where $\mathfrak{P}$ runs over the finitely many Hilbert partitions of $H_M(X)_{\preceq g}$, and selecting those for which there exist $m_i\in M$ for all $i\in I$ such that the condition (2) in Proposition \ref{hilbert:stanley} is fulfilled.
\end{coro}

The following example shows that Stanley decompositions induced by Hilbert partitions can effectively be computed.

\begin{ex}
We return to the Example \ref{example1}. Let $M=R \oplus (X_1,X_2)R$.
We consider again the Hilbert partitions
\begin{align*}
\mathfrak{P}_1:& (1+X_1+X_2+X_1X_2)+(X_1+X_1X_2)+X_2,\\
\mathfrak{P}_3:& (1+X_1+X_2+X_1X_2)+X_1+X_2+X_1X_2.
\end{align*}
It is easy to check that
\begin{align*}
\overline{\mathfrak{D}(\mathfrak{P}_1)}: M= & (1,0)K[X_1,X_2] \oplus (0,X_1)K[X_1,X_2] \oplus  (0,X_2) K[X_2],\\
\overline{\mathfrak{D}(\mathfrak{P}_3)}: M= & (1,0)K[X_1,X_2] \oplus (0,X_1X_2)K[X_1,X_2] \oplus (0,X_1)K[X_1] \oplus (0,X_2)K[X_2]
\end{align*}
are induced Stanley decompositions.
\end{ex}

\begin{rema} Corollary \ref{coro:last1} shows that the answer to Question \ref{Q:Herzog} is "Yes". However, in general the method presented above may be very difficult to use, since one has to test the conditions in Proposition
\ref{hilbert:stanley} (2) for all systems of elements $0\not = m_i\in M_{s_i}$. If we add the assumption that $\dim_K M_a \le 1$ for all $a\in \ZZ^n$, this simplifies considerably. We introduce below a precise algorithm for computing the Stanley depth in this case, solving Problem \ref{P:Herzog}.
\end{rema}

\begin{coro} \label{coro:last2}
Assume in addition that $\dim_K M_a \le 1$ for all $a\in \ZZ^n$. Then $\Stdepth M$ may be computed by the following Algorithm \ref{algorithm}.
\end{coro}

\begin{algorithm}[H]
\caption{Computing Stanley depth in case $\dim_K M_a \le 1$ for all $a\in \ZZ^n$}\label{algorithm}
(We assume that all $\beta_{0,a}=0$ unless $a \in \mathds{N}^n$)\;
Compute $g\in \ZZ^n$ such that $M$ is positively $g$-determined\;
Compute $H_M(X)_{\preceq g}$\;
\nl $j=n$\;
\nl \While {$j>0$} {
Compute $\mathrm{P}_j=\{\mathfrak{P}|\ \mathfrak{P} \text{ is Hilbert partition of } H_M(X)_{\preceq g}, \Hdepth \mathfrak{P}=j\}$\;
\ForAll {$\mathfrak{P}\in\mathrm{P}_j$}{
Set $\mathfrak{D}(\mathfrak{P}): M\iso\Dirsum_{i\in I} R_i(-s_i)$ be the induced Hilbert decomposition\;
$StanleyDecomposition=true$\;
\While{$i \in I$}{
\nl\If {$R_i\bigcap\Ann M_{s_i}\neq 0$}{
$StanleyDecomposition=false$;
}
}
\If {$StanleyDecomposition=true$}{
\nl \Return {j}\;
}
}

\nl $j=j-1$\;
}
\Return {0}\;
\end{algorithm}

\begin{proof} First we explain the algorithm. In line $\bf{1}$ the variable $j$ is initialized with the maximum possible value for the Stanley depth. In the loop starting at line $\bf{2}$ we are searching for a Stanley decomposition of depth $j$ which is induced by a Hilbert decomposition. If one of these is found then we return the value $j$ at line  $\bf{4}$ and we finish the search. If none is found then the value of $j$ is decreased at line $\bf{5}$.

The only fact to prove is that the condition at line $\bf{3}$ assures that $\mathfrak{D}(\mathfrak{P})$ is inducing a Stanley decomposition. Assume that for all $i\in I$ we have that $R_i\bigcap\Ann M_{s_i}= 0$. Let $0 \neq m_i\in M_{s_i}$. Since  $\dim_K M_{s_i} = 1$, we have $\Ann m_i=\Ann M_{s_i}$ and $m_iR_i=R_i(-s_i)$ as vector spaces (all the degrees are reached). Then
\[
M= \sum_{i\in I} m_i R_i
\]
as multigraded $R$-modules. The spaces $m_iR_i$ do not overlap because $\dim_K M_a \le 1$, so the sum is actually direct. Since each summand is a Stanley space, we obtain a Stanley decomposition.
\end{proof}

\section{Some applications}\label{Applications}

As shown in the previous sections, both the Hilbert depth and the Stanley depth of a finitely generated multigraded $R$-module $M$
can be computed by considering Hilbert partitions of the polynomial $H_M(X)_{\preceq g}$.
Note that these invariants cannot be easily computed in practice, since the number of possible partitions is huge (even in very simple cases, see e.g.~Example \ref{example1}).
In this section we will show that the methods introduced so far allow us however to deduce some simple statements.
For simplicity we shall make the same assumptions as in Section \ref{AlgHilb}.
\medskip

The following proposition was proved in \cite[Lemma 3.6]{HVZ} for ideals. Now we can state and prove it for Stanley depth of modules.

\begin{propo}Let $M$ be a finitely generated multigraded $R$-module.  Let $R'=R\otimes_{K} K[X_{n+1},\ldots,X_{n+m}]$ be the polynomial ring in $n+m$ variables and $M'=M\otimes_{K} K[X_{n+1},\ldots,X_{n+m}]$ the module obtained from $M$ by scalar extension. Then
\begin{enumerate}
\item $\depth_{R'}M'=\depth_{R}M+m$;
\item $\Hdepth_{R'}M'=\Hdepth_{R}M+m$;
\item $\Stdepth_{R'}M'=\Stdepth_{R}M+m$.
\end{enumerate}
\end{propo}

\begin{proof} The statement about $\depth$ is clear since $X_{n+1},\ldots,X_{n+m}$ is a regular sequence for $M'$. Assume that the module $M$ is positively $g$-determined and set
$g'=(g,0,\ldots,0)\in \ZZ^{n+m}$. Since the multiplication map $\cdot X_i : M_{a} \longrightarrow M_{a+e_i}$ is an isomorphism whenever $i \ge n+1$, we deduce that $M'$ is positively $g'$-determined.
It follows
$$H_M(X)_{\preceq g}=H_{M'}(X)_{\preceq g'}=:P(X),$$
from which we  deduce the statement about $\Hdepth$ using Theorem \ref{theo:main} and Corollary \ref{cor:hdepth}. By Corollary \ref{coro:last1} and Proposition \ref{hilbert:stanley} it is clear that a Hilbert partition of $P(X)$ is inducing a Stanley decomposition for $M$ if and only if it is inducing a Stanley decomposition for $M'$, and the last statement follows.
\end{proof}

In the same fashion as above, the following proposition---shown in \cite[Corollary 3.2]{N} and \cite[Lemma 2.2]{C} for the case of Stanley depth of ideals---can be also extended to modules.

\begin{propo}\label{Prop:apl2} Let $R'=R\otimes_{K} K[X_{n+1},\ldots,X_{n+m}]$ be the polynomial ring in $n+m$ variables, and let $M'$ be a finitely generated multigraded $R'$-module. We consider
$\phi:R'\rightarrow R$, $\phi(X_i)=X_i$ for $i\le n$ and $\phi(X_i)=1$ for $n< i$.
Let $M$ be the $R$-module obtained from $M'$ via $\phi$. Then
\begin{enumerate}
\item $\depth_{R'}M'\le\depth_{R}M+m$;
\item $\Hdepth_{R'}M'\le\Hdepth_{R}M+m$;
\item $\Stdepth_{R'}M'\le\Stdepth_{R}M+m$.
\end{enumerate}
\end{propo}

\begin{proof} The statement about $\depth$ is easy since every resolution of $M'$ over $R'$ induces a free resolution of $M$ over $R$ via $\phi$, and we can then use the Auslander--Buchsbaum formula.

Further, observe that
$$
H_M(X_1,\dots,X_n)=H_{M'}(X_1,\dots,X_n,1,\ldots,1).
$$
Assume that the module $M'$ is positively $g'=(g_1,\ldots , g_n,g_{n+1}, \ldots , g_{m+n})$-determined and set
$g=(g_1,\ldots , g_n)\in \ZZ^{n}$. We deduce that $M$ is positively $g$-determined and
$$
H_M(X_1,\dots,X_n)_{\preceq g}=H_{M'}(X_1,\dots,X_n,1,\ldots,1)_{\preceq g'}.
$$
Given $a'=(a_1,\ldots , a_n,a_{n+1}, \ldots , a_{m+n}),b'=(b_1,\ldots , b_n,b_{n+1}, \ldots , b_{m+n}) \in \ZZ^{n+m}$ such that $a'\preceq b'$, we consider the polynomial induced by the interval $[a',b']$, namely
\[
Q[a',b'](X)=\sum_{a' \preceq c \preceq b' } X^{c}.
\]
It is easy to check that
\[
Q[a',b'](X_1,\dots,X_n,1,\ldots,1)=\Big(\prod_{n+1\le i \le n+m } (b_i-a_i+1)\Big)Q[a,b](X_1,\dots,X_n)
\]
for $a=(a_1,\ldots , a_{n})$ and $b=(b_1,\ldots , b_{n})$. We conclude that each Hilbert partition of $H_{M'}(X)_{\preceq g'}$ with polynomials induced by intervals of type $[a',b']$ determines a Hilbert partition of $H_{M}(X)_{\preceq g}$
with polynomials induced by intervals of type $[a,b]$, and
the statement about $\Hdepth$ is deduced by using Theorem \ref{theo:main}.

By  Proposition \ref{hilbert:stanley} it is clear that a Hilbert partition of $H_{M'}(X)_{\preceq g'}$ is inducing a Stanley decomposition for $M'$ only if it is inducing a Stanley decomposition for $M'$, since there are less linear dependencies to check. The last statement on $\Stdepth$ follows straight.
\end{proof}

\begin{ex} The inequalities in Proposition \ref{Prop:apl2} may be strict. Let $M'$ be the ideal $(XY,XZ)$ in $R'=K[X,Y,Z]$. It is clear that $\depth_{R'}M'<3,\ \Hdepth_{R'}M'<3,\ \Stdepth_{R'}M'<3$ since $M'$ is not principal. Set $n=1$, $m=2$ and $\phi: K[X,Y,Z]\rightarrow K[X]$. Then $M$ is the principal ideal $(X)$ in $R=K[X]$ and  $\depth_{R}M=\Hdepth_{R}M=\Stdepth_{R}M=1$.
\end{ex}

\section*{acknowledgement}
The authors would like to thank Winfried Bruns and Marius Vladoiu for their useful comments.

The first author was partially supported by CNCSIS grant TE-46 nr. 83/2010, and
the second author was partially supported by the Spanish Government through Ministerio de Educaci\'on y Ciencia (MEC), grant MTM2007-64704, and Ministerio de Econom\'ia y Competitividad, grant MTM2012--36917--C03--03, in cooperation with the European Union in the framework of the founds ``FEDER'', during the preparation of this work.

\vspace{\baselineskip}


\begin{thebibliography}{10}
\bibitem{A1}  J. Apel, On a conjecture of R. P. Stanley, Part I-Monomial Ideals,  J.~Algebr.~Comb., 17, 39--56 (2003)
\bibitem{A2}   J. Apel, On a conjecture of R. P. Stanley, Part II-Quotients Modulo Monomial Ideals, J.~Algebr.~Comb., 17, 57--74 (2003)
\bibitem{AP}  I. Anwar and D. Popescu, Stanley conjecture in small embedding dimension, J. ~Algebra, 318, 1027--1031 (2007)

\bibitem{BG}
W.~Bruns and J.~Gubeladze, Polytopes, Rings and K-Theory,
Springer (2009)

\bibitem{BKU}  W.~Bruns, Chr.~Krattenthaler, and J.~Uliczka, Stanley decompositions and Hilbert depth in the Koszul complex,
               J.~Comm.~Alg., 2, 327--357 (2010)

\bibitem{BKU2}  W.~Bruns, Chr.~Krattenthaler, and J.~Uliczka, Hilbert depth of powers of the maximal ideal,
               Contemp.~Math.,~vol.~555 (2011)

\bibitem{C} M.~Cimpoeas, Stanley depth of complete intersection monomial ideals, Bull. Math. Roumanie, 51, 205--211 (2008)

\bibitem{H} J.~Herzog, A survey on Stanley depth. In ``Monomial Ideals, Computations and
Applications", A.~Bigatti, P.~Gim\'enez, E.~S\'aenz-de-Cabez\'on (Eds.), Proceedings of MONICA 2011. Lecture Notes in Math. 2083, Springer (2013).

\bibitem{HP}  J.~Herzog and D.~Popescu, Finite filtrations of modules and shellable multicomplexes, Manuscr.~math., 121, 385--410 (2006)

\bibitem{HSY}
J.~Herzog, A.~S.~Jahan and S.~Yassemi, Stanley
decompositions and partitionable simplicial complexes,
J.~Algebr.~Comb., 27, 113--125 (2008)

\bibitem{HVZ}
J.~Herzog, M.~Vladoiu, X.~Zheng, How to compute the Stanley depth of a monomial ideal,
J.~Algebra, 322, 3151--3169 (2009)

\bibitem{JU}
J.J.~Moyano-Fern\'andez and J.~Uliczka, Hilbert depth of graded modules over polynomial rings in two variables,
               J.~Algebra, 373, 130--152 (2013)

\bibitem{M}
E.~Miller, The Alexander duality functors and local duality with monomial support,
J.~Algebra, 231, 180--234 (2000)

\bibitem{N}
S.~Nasir, Stanley decompositions and localization, Bull. Math.
Roumanie, 51, 151--158 (2008)

\bibitem{P}
D.~Popescu, Stanley depth of multigraded modules, J.~Algebra,
312,~2782--2797 (2009)

\bibitem{S}
R.~P.~Stanley, Linear Diophantine equations and local
cohomology, Invent. math., 68, 175--193 (1982)

\bibitem{U}  J.~Uliczka, Remarks on Hilbert Series of Graded Modules over Polynomial Rings,
              Manuscr.~math., 132, 159--168 (2010)
\end{thebibliography}
\end{document}